\documentclass{article}

\usepackage{graphicx}
\usepackage{latexsym}
\usepackage[fleqn]{amsmath}
\usepackage{amssymb}
\usepackage{amsbsy}
\usepackage{nishimura-pp}

\title{Stochastic Asymptotic Stabilizers for Deterministic Input-Affine Systems based on Stochastic Control Lyapunov Functions\thanks{This work was partially supported by Grant-in-Aid for Young Scientists (B) of KAKENHI (22760320) and Grant-in-Aid for Scientific Research (B) of KAKENHI (22360167). E-mail: yunishi@yamaguchi-u.ac.jp}}

\author{Y\^uki Nishimura\footnotemark[2]
\and Kanya Tanaka\footnotemark[2]
\and Yuji Wakasa\footnotemark[2]
\and Yuh Yamashita\footnotemark[5]}

\begin{document}
\maketitle

\renewcommand{\thefootnote}{\fnsymbol{footnote}}

\footnotetext[2]{Yamaguchi University}
\footnotetext[5]{Hokkaido University}

\renewcommand{\thefootnote}{\arabic{footnote}}

\hyphenation{Lyapunov}
\hyphenation{Markov}
\hyphenation{Alcaraz}
\hyphenation{Kushner}
\hyphenation{Krasovskii}
\hyphenation{Schultz}
\hyphenation{Zubov}
\hyphenation{Vannelli}
\hyphenation{Vidyasagar}
\hyphenation{Bucy}
\hyphenation{Mao}
\hyphenation{Hamiltonian}
\hyphenation{Lagrangian}
\hyphenation{Hess}
\hyphenation{Euclid}
\hyphenation{Poisson}
\hyphenation{Wiener}
\hyphenation{Nishimura}
\hyphenation{Yamashita}
\hyphenation{Gaussian}
\hyphenation{Brockett}

\maketitle

\begin{abstract}
In this paper, a stochastic asymptotic stabilization method is proposed for deterministic input-affine control systems, which are randomized by including Gaussian white noises in control inputs. The sufficient condition is derived for the diffucion coefficients so that there exist stochastic control Lyapunov functions for the systems. To illustrate the usefulness of the sufficient condition, the authors propose the stochastic continuous feedback law, which makes the origin of the Brockett integrator become globally asymptotically stable in probability.
\end{abstract}


%

\pagestyle{myheadings}
\thispagestyle{plain}
\markboth{Y. Nishimura, K. Tanaka, Y. Wakasa, and Y. Yamashita}{Stochastic Asymptotic Stabilizers for Deterministic Input Affine Systems}

\section{Introduction}

This paper proposes design concepts of the diffusion coefficients and the stochastic control Lyapunov functions for stochastic stabilization problems of general deterministic input-affine control systems.

In deterministic control problems, such as nonholonomic systems, there exist systems--- not locally asymptotically stabilizable using any continuous time-invariant feedback law ---which are controllable \cite{coron,zabczyk}. For such systems, previous workers proposed different control laws: the time-varying feedback laws \cite{fujimoto,sordalen,tian}, the discontinuous feedback laws \cite{astolfi,bloch}, variable constraint control laws \cite{tamura,yang}, and time-state control laws \cite{sampei1994}. For the deterministic systems, on the other hand, the authors of this paper propose the stochastic feedback laws via the stochastic control Lyapunov functions \cite{nishimura2009cdc}. The aim of the authors' previous work has been to design continuous stochastic feedback laws; however, the continuity of the proposed controllers was not investigated. Moreover, the randomizing problems were not well-considered; i.e., Wong-Zakai theorem \cite{wong} was not considered when the deterministic systems were randomized. 

In this paper, the general deterministic input-affine control systems are randomized using the Wong-Zakai theorem. Then, sufficient conditions for diffusion coefficients are derived so that there exist stochastic control Lyapunov functions for the systems. Further, the stochastic continuous feedback law is proposed as it enables the nonholonomic system become globally asymptotically stable in probability.

This paper is organized as follows. In \ssref{subsec:idea}, the motivation for this reserach is described using Brockett integrator \cite{brockett}, a typical nonholonomic system. In \sref{sec:preliminary}, the basic results of stochastic stabilities, stochastic stabilizabilities, and randomization problems are summarized. In \sref{sec:main}, the main results of this paper are presented. In \ssref{subsec:general}, general deterministic input-affine control systems are randomized via the Wong-Zakai theorems, besides considering the design strategies of the diffusion coefficients and the stochastic control Lyapunov functions. In \ssref{subsec:brockett}, the validity of the strategies is confirmed by obtaining a Sontag-type \cite{sontag} stochastic controller for the Brockett integrator. The sufficient condition for the proposed controller to be continuous is also obtained. Moreover, it is proven that the origin of the resulting closed-loop system is globally asymptotically stable in probability. \sref{sec:simulation} shows the numerical simulation of the Brockett integrator with the proposed controller. \sref{sec:conclusion} concludes the paper.

In this paper, $\R^n$ denotes an $n$-dimensional Euclidean space. For a vector $x \in \R^n$ and a mappings $g: \R^n \rightarrow \R^n \times \R^r$ and $V: \R^n \rightarrow \R$, the Lie derivative of $V(x)$ is represented by
\begin{align}
\lie{g}{V}(x) := \pfrac{V(x)}{x} g(x).
\end{align}
The conditional probability of some event $A$, under the condition $B$, is written $\pr \{ A |B\}$. A one-dimensional standard Wiener process is represented as $w(t) \in \R$. For $\sigma: \R^n \rightarrow \R$, the differential forms of the Stratonovich and Ito integrals in $w(t)$ are denoted by $\sigma(x) \circ dw$ and $\sigma(x) dw$, respectively.

\begin{remark}
In this paper, one-dimensional Wiener processes are used for randomizing, because the aim of this paper is to design continuous feedback laws. If multi-dimensional Wiener processes $W(t) \in \R^d$ with $d \ge 2$ are applied, the Ito mappings $\Phi:[0,\infty) \times \R^d \rightarrow \R^n: (W,t) \mapsto x$ of the randomized systems $x(t)=\Phi(W,t)$ become discontinuous with probability one \cite{ikeda}.
\end{remark}

\subsection{Motivation}\label{subsec:idea}

In control problems for deterministic nonlinear systems, such as nonholonomic systems, there exist systems which are not locally asymptotically stabilizable by using any continuous state feedback law, although they are controllable \cite{coron}. The Brockett integrator
\begin{align}\label{eq:brockett}
\hspace{-0.5cm}\dot{x} = g(x) u_c = \begin{bmatrix}
g_1(x) \\g_2(x)
\end{bmatrix}u_c = \begin{bmatrix}
b_1 & 0 \\ 0 & b_2 \\ b_3 x_2 & -b_4 x_1
\end{bmatrix} \begin{bmatrix}
u_{c_1} \\ u_{c_2}
\end{bmatrix}
\end{align}
is one of typical nonholonomic systems\footnote{The original Brockett integrator \cite{brockett} satisfies $b_1=b_2=b_3=b_4=1$. Further, if $b_1=b_2=b_3=1$ and $b_4=0$, the system \eqref{eq:brockett} is said to be a chained system \cite{nijimeijer}.}, where $b_1,b_2,b_3,b_4 \in \R$ satisfy $b_1\neq 0$, $b_2 \neq 0$, and $b_1b_4+b_2b_3 \neq 0$. This system is controllable because the rank of a matrix
\begin{align}
&\begin{bmatrix}
g_1(x) & g_2(x) & \lie{g_1}{g_2}(x)-\lie{g_2}{g_1}(x)\end{bmatrix} = \begin{bmatrix}
b_1 & 0 & b_3 x_2 \\ 0 & b_2 & -b_4 x_1 \\ 0 & 0 & b_2 b_3 + b_1 b_4
\end{bmatrix}
\end{align}
is $3$ for all $x \in \R^3$ \cite{nijimeijer}; however, the system does not have any continuous feedback stabilizer, because it is a driftless affine nonholonomic system \cite{coron}. 

Let the stabilization problem of the Brockett integrator \eqref{eq:brockett} be considered from the viewpoint of the control Lyapunov theory \cite{sontag}. A simple positive definite function $V_1(x):=x_1^2+x_2^2+x_3^2$ is not a control Lyapunov function for \eqref{eq:brockett}, because
\begin{align}
\dot{V}_1(x) = \lie{g}{V_1}(x) u = 0
\end{align}
is derived for $x \in M:=\{x \in \R^3 |x_1=x_2=0\}$. Hence, one has to consider a new positive definite function $V_2: \R^n \rightarrow \R$, which is concave down in $M$; i.e.,
\begin{align}\label{eq:sclf-bro}
&V_2(x) := 2 x_3^2 - \frac12 X(x)(1+x_3^2)+2 \left|\frac{X(x)}{2}\right|^{1+\frac{x_3^2}{2}},
\end{align}
where
\begin{align}
&X(x) := x_1^2+x_2^2.
\end{align}
Fig.~\ref{fig:sclf} implies that if the initial state is not in $M \backslash \{0\}$, and if the feedback control law is designed so that $\dot{V}<0$ except $M$, then the trajectories of the system \eqref{eq:brockett} converge to the origin. However, the origin of the resulting closed-loop system is not locally asymptotically stable, because any neighborhood of the origin contains the nonempty subset of $M \backslash \{0 \}$. On the other hand, the Hessian of $V_2(x)$ in $M$ is calculated as
\begin{align}
\hspace{-0.5cm}\left[ \pfrac{}{x}\left[ \pfrac{V_2}{x} \right]^T \right](x)= \begin{bmatrix}
-(1+x_3^2) & 0 & 0 \\ 0 & -(1+x_3^2) & 0 \\ 0 & 0 & 4
\end{bmatrix},
\end{align}
which has negative-valued eigenvalues. This implies that the trajectory starting from $M$ possibly converges to the origin, if the Hessian has a role in the flow of $V_2(x)$.

In this paper, to use the Hessian effectively, feedback control laws involving the Wiener processes are considerd. In other words, the aim of this paper is to solve the stochastic stabilization problems of the deterministic input-affine control systems by using stochastic control Lyapunov functions. 

\begin{remark}
The foregoing approach is analogous to the globally asymptotically stabilization problems for systems with non-contractible state space \cite{tsuzuki}. However, the basic idea of this paper is different from that of \cite{tsuzuki} because, this paper proposes stochastic continuous feedback laws and \cite{tsuzuki} deterministic discontinuous feedback laws. \eot
\end{remark}

\begin{figure}[t]
\begin{center}
\includegraphics[width=0.4\textwidth,keepaspectratio=true]{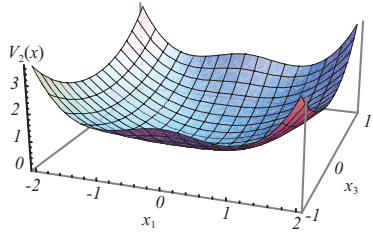}
\end{center}
\caption{Positive Definite Function $V_2(x)$ in $x_2=0$}
\label{fig:sclf}
\end{figure}

\section{Preliminary Discussion}\label{sec:preliminary}

In this section, the basic results of stochastic stabilities, stochastic stabilizabilities, and randomization problems are summarized and discussed.

\subsection{Stochastic Control Lyapunov Theory}

In this paper, the theories of Lyapunov stability and stabilizability are used. In this subsection, the previous results of Hasminskii \cite{hasminskii} and Florchinger \cite{florchinger1995,florchinger1997feedback,florchinger1997global} are described.

Let a stochastic system
\begin{align}\label{eq:sys-sto}
dx(t) = f(x(t)) dt + \sigma(x(t)) dw(t)
\end{align}
and a control stochastic system
\begin{align}\label{eq:sys-sto-ctrl}
dx(t) = \{ f(x(t))+g(x(t))u \} dt + \sigma(x(t)) dw(t)
\end{align}
be considered, where the initial state $x(0)=x_0$ is given in $\R^n$; $u: \R^n \rightarrow \R^m$ is a control input; $f: \R^n \rightarrow \R^n$, $\sigma: \R^n \rightarrow \R^n \times \R$, and $g: \R^n \rightarrow \R^n \times \R^m$ are Lipshitz functionals satisfying $f(0)=0$ and $\sigma(0)=0$. Also, it is assumed that a non-negative constant $K$ satysfying
\begin{align}
|f(x)| + |g(x)| + |h(x)| \le K (1 +|x|)
\end{align}
exists for any $x \in \R^n$.

\begin{definition}[stability in probability \cite{florchinger1995,hasminskii}]
The equilibrium $x(t)\equiv 0$ of the system \eqref{eq:sys-sto} is stable in probability if 
\begin{align}
\lim_{x_s \rightarrow 0} \pr \left\{ \sup_{s \le t} |x(t)|>\epsilon \Big| x(s)=x_s \right\} = 0
\end{align}
for any $s \ge 0$ and any $\epsilon > 0$. \eod
\end{definition}
\begin{definition}[local asymptotic stability in probability \cite{florchinger1995,hasminskii}]\label{def:lasip}
The equilibrium $x(t)\equiv 0$ of the system \eqref{eq:sys-sto} is locally asymptotically stable in probability if $x(t) \equiv 0$ is stable in probability and
\begin{align}
\lim_{x_s \rightarrow 0} \pr \left\{ \lim_{t \rightarrow +\infty} |x(t)|=0 \Big| x(s)=x_s \right\} = 1
\end{align}
for any $s \ge 0$. \eod
\end{definition}
\begin{definition}[global asymptotic stability in probability \cite{florchinger1997global,hasminskii}]
The equilibrium $x(t)\equiv 0$ of the system \eqref{eq:sys-sto} is locally asymptotically stable in probability if $x(t) \equiv 0$ is stable in probability and
\begin{align}
\pr \left\{ \lim_{t \rightarrow +\infty} |x(t)|=0 \Big| x(s)=x_s \right\} = 1
\end{align}
for any $s \ge 0$ and $x \in \R^n$. \eod
\end{definition}

\begin{definition}[local asymptotic stabilizability in probability \cite{florchinger1995}]
The equilibrium $x(t)\equiv 0$ of the system \eqref{eq:sys-sto-ctrl} is locally asymptotically stabilizable in probability if there exists a neighborhood $D \in \R^n$ of the origin and a function $k : D \rightarrow \R^n$ such that the solution of the closed-loop system
\begin{align}\label{eq:closed}
dx = (f(x)+g(x)k(x))dt + \sigma(x) dw
\end{align}
is uniquely defined in $x \in D$, and the equilibrium $x(t) \equiv 0$ of \eqref{eq:closed} is locally asymptotically stable in probability. \eod
\end{definition}

Let the {\it infinitesimal generators} be defined
\begin{align}
\mathcal{L}(\cdot) = &\lie{f}{(\cdot)}(x) + \frac12 \sigma(x)^T \left[ \pfrac{}{x}\left[\pfrac{(\cdot)}{x}\right]^T \right] \sigma(x)
\end{align}
for the system \eqref{eq:sys-sto} and 
\begin{align}
\mathcal{L}(\cdot) = &\lie{f}{(\cdot)}(x) + \lie{g}{(\cdot)}(x) u + \frac12 \sigma(x)^T \left[ \pfrac{}{x}\left[\pfrac{(\cdot)}{x}\right]^T \right] \sigma(x)
\end{align}
for the system \eqref{eq:sys-sto-ctrl}.

\begin{definition}[stochastic control Lyapunov function \cite{florchinger1995}]\label{def:sclf}
A function $V: D \subset \R^n \rightarrow \R$ is said to be a stochastic control Lyapunov function of the system \eqref{eq:sys-sto-ctrl} if $V(x)$ is twice differentiable in $x \in D$, proper in $D$, and $(\mathcal{L}V)(x)$ is negative definite in $x \in D \cap \{ \lie{g}{V}(x) = 0 \}$. \eod
\end{definition}

\begin{definition}[bounded control property \cite{florchinger1995}]\label{def:sbcp}
A stochastic control Lyapunov function $V: D \rightarrow \R$ is said to satisfy the bounded control property if there exists a function $d: D \rightarrow (0,\infty)$ such that $d$ is bounded on $D$ and for every $x \in D \backslash \{0\}$ there exists a control $u \in \R^m$ such that $||u|| < d(x)$ and $(\mathcal{L}V)(x) < 0$.
\end{definition}

\begin{definition}[small control property \cite{florchinger1995}]\label{def:sscp}
A stochastic control Lyapunov function $V: D \rightarrow \R$ is said to satisfy the small control property if $V(x)$ satisfies the bounded control property and $\lim_{x \rightarrow 0} d(x) = 0$. \eod
\end{definition}

Using the above definitions, the following theorems are obtained.

\begin{theorem}[Florchinger \cite{florchinger1997feedback}, Theorem 1.2]\label{the:florchinger}
$\quad$
\begin{enumerate}
\item There exists a stochastic control Lyapunov function $V(x)$ of the system \eqref{eq:sys-sto-ctrl} if and only if the system \eqref{eq:sys-sto-ctrl} is locally asymptotically stablilzable in probability by means of a feedback law $u=k(x)$ which is smooth in $D \backslash \{0\}$.
\item The stochastic control Lyapunov function $V(x)$ in 1 satisfies the bounded control property if and only if the feedback laws $k(x)$ in 1 satisfy $||k(x)||<d(x)$ for every $x \in D \backslash \{0\}$.
\item If the feedback law $k(x)$ in 2 is bounded in $D$, and if the stochastic control Lyapunov function $V(x)$ in 1 satisfies the small control property, then $\lim_{x \rightarrow 0} k(x) = 0$.
\end{enumerate}
\eot
\end{theorem}

\begin{theorem}[Hasminskii \cite{hasminskii}]\label{the:hasminskii}
The equilibrium $x(t) \equiv 0$ of \eqref{eq:sys-sto} is globally asymptotically stable in probability if there exists a function $V: \R^n \rightarrow \R$ which is twice continuous in $x \in \R^n$ and proper in $x \in \R^n$ such that $(\mathcal{L}V)(x)$ is negative definite in $x \in \R^n$. \eot
\end{theorem}

\tref{the:florchinger} implies that if there exists a stochastic control Lyapunov function $V: \R^n \rightarrow \R$ of the system \eqref{eq:sys-sto-ctrl}, and if $V(x)$ satisfies the small control property, then there exists a continuous feedback law which makes the origin of the system \eqref{eq:sys-sto-ctrl} locally asymptotically stable in probability. Moreover, \tref{the:hasminskii} yields that the resulting closed-loop system is globally asymptotically stable in probability if the stochastic control Lyapunov function $V(x)$ is proper in $\R^n$.

\begin{remark}\label{rem:kushner}
Because the deterministic systems are considered in this paper, the paths of the stochastic systems formed \eqref{eq:sys-sto} should converge to the origin with probability one. Although \dref{def:lasip} does not clarify the probabilities that the paths converge to the origin, the other definition by Kushner \cite{kushner} claims that if the function $V: \R^n \rightarrow \R$ is proper and positive definite in $Q_m:= \{x \in \R^n | V(x)<m\}$, and if $(\mathcal{L}V)(x)$ is negative definite in $Q_m$, then the paths starting from $x(0)=x_0 \in Q_m$ converge to the origin at least with probability $1-V(x_0)/m$. The fact yields that if there exists $V(x)$ satisfying the conditions of  \tref{the:hasminskii}, then the paths of the stochastic system \eqref{eq:sys-sto} converge to the origin with probability one. One has to note that this discussion is valid for the systems with Euclidean state spaces. \eot
\end{remark}

\subsection{Wong-Zakai Approximation Theorem}

To randomize the deterministic systems, the Wong-Zakai correction terms \cite{twardowska,wong} should be considered. This subsection shows the Wong-Zakai approximation theorem and discusses the necessity of the correction term.

Let a sequence
\begin{align}
a=t^{(n)}_0 < t^{(n)}_1 < \ldots < t^{(n)}_n = b,\ t \in [a,b] \subset \R,
\end{align}
and a piecewise linear sequence 
\begin{align}\label{eq:wz-wiener}
\hspace{-0.5cm}w^{(n)}(t) := w^{(n)}(t_i)+\frac{w^{(n)}(t_{i+1})-w^{(n)}(t_i)}{t^{(n)}_{i+1}-t^{(n)}_i}(t-t^{(n)}_i)
\end{align}
be considered, where $t^{(n)}_{i+1} \ge t \ge t^{(n)}_i$ and $0 \le i \le n-1$. For a system
\begin{align}\label{eq:wongzakai}
dx^{(n)}(t) = f(x^{(n)}(t),t) dt + \sigma(x^{(n)},t) dw^{(n)}(t),
\end{align}
the following assumptions are considered:
\begin{description}
\item[$(H_1)$] $\sigma'(x,t) := \partial \sigma(x,t)/\partial x$ are continuous in $x$ and $t$.
\item[$(H_2)$] $\sigma$, $\sigma'\sigma$, and $f$ are continuous in $t$.
\item[$(H_3)$] There exist constants $K$ satisfying $|\sigma(x,t)-\sigma(x_0,t)| \le K |x-x_0|$, $|\sigma'(x,t)\sigma(x,t)-\sigma'(x_0,t)\sigma(x_0,t)| \le K |x-x_0|$, and $|f(x,t)-f(x_0,t)| \le K |x-x_0|$.
\end{description}

Then, the following theorem is derived.
\begin{theorem}[Wong and Zakai \cite{wong}]\label{the:wong}
If the system \eqref{eq:wongzakai} satisfies the assumptions $(H_1)$--$(H_3)$, then, the system \eqref{eq:wongzakai} converges in the mean to the system
\begin{align}\label{eq:sys-wong}
dx(t) = &f(x(t),t) dt + \sigma(x(t),t) dw(t) + \frac12 \sigma'(x(t),t) \sigma(x(t),t) dt
\end{align}
as $n \rightarrow 0$.
\eot
\end{theorem}

\tref{the:wong} implies that if the Wiener process $w(t)$ is applied, then the third term of the right-hand side of equation \eqref{eq:sys-wong} is generated. In other words, if the Wiener process $w(t)$ is applied, then the stochastic integrals in $w(t)$ should be defined by Stratonovich integrals \cite{twardowska}.

\begin{remark}
In other randomizing problems, there are cases that Ito integrals are valid for stochastic integrals \cite{oksendal}. Nevertheless, Stratonovich integrals are considered reasonable for this study. The reason is that the sequence \eqref{eq:wz-wiener} is used as an approximation of the Wiener process, because the strict white noises are impossible to generate in practice. Moreover, in the randomization problems, Ito and Stratonovich integrals provide different results; for example, let a one-dimensional system
\begin{align}\label{eq:sys-wong2}
\dot{x}(t) = x \dot{\xi}(t),\quad x(t) \in \R,
\end{align}
be considered, where a stochastic noise $\xi(t) \in \R$ is once differentiable in $t$. If Ito integrals are employed, $dx(t) = x(t) dw(t)$ is obtained replacing $\xi(t)$ by $w(t)$; the solution to this equation is $x(t) = x(0) e^{-t/2 + w(t)}$. However, the process is not a solution to \eqref{eq:sys-wong2} if $w(t)$ is re-replaced by $\xi(t)$. In contrast, if Stratonovich integrals are employed, one obtains $dx(t) = x(t) \circ dw(t) = x(t)dw(t) + (1/2) x(t) dt$. The solution to this equation is $x(t) = x(0) e^{w(t)}$, which is also the solution to \eqref{eq:sys-wong2}, if $w(t)$ is re-replaced by $\xi(t)$. That explains why Stratonovich integrals are employed in this paper. \eot
\end{remark}

\section{Stochastic Stabilization for Deterministic Systems}\label{sec:main}

This section presents the main results of this paper. In the next subsection, a sufficient condition is derived so that the origin of a deterministic system
\begin{align}\label{eq:the-det}
\dot{x}(t) = f(x(t)) + g(x(t)) u(t)
\end{align}
is locally asymptotically stabilizable in probability by means of a feedback law, including a Gaussian white noise. Moreover in \ssref{subsec:brockett}, a continuous stochastic feedback stabilizer is proposed for the Brockett integrator \eqref{eq:brockett}.

\subsection{Sufficient Conditon for Stochastic Stabilizers}\label{subsec:general}

In randomizing problems, the diffusion coefficients of the Wiener processes play a crucial role for stochastic stabilities. The design concepts of the diffusion coefficients and the stochastic control Lyapunov functions are considered by using the following theorem.

\begin{theorem}\label{the:general}
Let deterministic system \eqref{eq:the-det} and a control input
\begin{align}\label{eq:the-ctrl}
u(t) dt = v(t) dt + B(x(t))\circ dw(t)
\end{align}
be considered, where $v(t) \in \R^m$ is a new input vector and $B:\R^n \rightarrow \R^m$ is (a part of) the diffusion coefficient. If $B(x)$ is once differentiable, $B(0)=0$, and there exists a twice differentiable positive definite function $V: \R^n \rightarrow \R$ such that
\begin{align}\label{eq:the-sclf}
\frac12 B(x)^T g(x)^T &\left[ \pfrac{}{x}\left[\pfrac{V(x)}{x}\right]^T \right] g(x) B(x) \nonumber \\
&+ \frac12\pfrac{V(x)}{x} \pfrac{g(x) B(x)}{x} g(x) B(x)< -\lie{f}{V}(x)
\end{align}
for all $x \in \{ x \in \R^n \backslash \{0\} | \lie{g}{V}(x)=0 \}$, then $V(x)$ is a stochastic control Lyapunov function of the system \eqref{eq:the-det} with \eqref{eq:the-ctrl}. \eot
\end{theorem}

\begin{proof}
Randomizing the system \eqref{eq:the-det} by using \eqref{eq:the-ctrl}, the following stochastic system
\begin{align}\label{eq:the-sto}
\hspace{-0.5cm}dx = \left\{ g(x)v + f(x)+\frac12 \pfrac{B(x) g(x)}{x} B(x) g(x)  \right\} dt + g(x)B(x) dw
\end{align}
is obtained. By \dref{def:sclf}, \eqref{eq:the-sclf}, $B(0)=0$, and $B(x) \in C^1$, the theorem is derived.
\end{proof}

\tref{the:general} is almost trivial if \dref{def:sclf} is considered; however, the theorem clarifies that the deterministic system \eqref{eq:the-det} may be asymptotically stabilized in probability using suitably-designed diffusion coefficient $B(x)$ and stochastic control Lyapunov function $V(x)$. The design concepts of $B(x)$ and $V(x)$ are considered as follows. If $V(x)$ is designed so that the eigenvalues of 
\begin{align}
g^T(x) \left[\pfrac{}{x}\left[\pfrac{V(x)}{x}\right]^T\right] g(x)
\end{align}
are negative while $\lie{g}{V}(x) = 0$, then the first term of the left-hand side of \eqref{eq:the-sclf} is negative. To satisfy this condition, $V(x)$ should be concave down in $\lie{g}{V}(x)=0$. Further, $B(x)$ should be designed so that the sufficient condition \eqref{eq:the-sclf} is satisfied; e.g., if $\lie{f}{V}(x)=0$ while $\lie{g}{V}(x)=0$, then $B(x)$ is designed so that the second term of the left-hand side of \eqref{eq:the-sclf} is negative.

In addition, if a stochastic control Lyapunov function satisfies the small control property, there is a continuous stochastic feedback law. The next subsection confirms that the design concepts can be used effectively by obtaining a continuous stochastic feedback stabilizer for the Brockett Integrator \eqref{eq:brockett}.

\subsection{Case Study: Continuous Stochastic Stabilizer for Brockett Integrator} \label{subsec:brockett}

This subsection is a continuation of the discussion of \ssref{subsec:idea}. First, in \tref{the:brockett}, one demonstrates the existence of the diffusion coefficients such that $V_2(x)$ becomes a stochastic control Lyapunov function for the Brockett integrator \eqref{eq:brockett}. Second, in \tref{the:sontag}, a Sontag-type feedback controller is derived using the stochastic control Lyapunov function $V_2(x)$. Third, in \tref{the:sscp}, a sufficient condition is obtained such that the Sontag-type controller is continuous for all $x \in \R^3$. Finally, in \corref{cor:global}, it is proved that the proposed stabilizer makes the origin of the resulting closed-loop system globally asymptotically stable in probability.

\begin{theorem}\label{the:brockett}
Let the Brockett Integrator \eqref{eq:brockett} and a control input
\begin{align}\label{eq:brockett2}
\hspace{-0.5cm} u_c(t) dt= \{ u(t) + v(x(t)) \}dt + B(x(t)) \circ dw(t),
\end{align}
be considered, where
\begin{align}
\label{eq:brockett3}v(x) &:= \begin{bmatrix}
-\dfrac{1}{2b_1} \pfrac{\sigma_1}{x}(x) \sigma(x) \\
-\dfrac{1}{2b_2} \pfrac{\sigma_2}{x}(x) \sigma(x)
\end{bmatrix},\\
\label{eq:brockett4}\sigma(x) &:= \begin{bmatrix}
\sigma_1(x) & \sigma_2(x) & \sigma_3(x)
\end{bmatrix}^T = g(x) B(x),\\
\label{eq:brockett5}B(x) &:= \begin{bmatrix}
B_1(x) & B_2(x)
\end{bmatrix}^T.
\end{align}
If the diffusion coefficient $B(x)$ is once differentiable in $x$ and satisfies
\begin{align}
\label{eq:brockett6}B_1(0)&=B_2(0)=0, \\
\label{eq:brockett7}B_1(x)&B_2(x)(b_1b_4-b_2b_3)x_3 \ge 0,\\
\label{eq:brockett8}B_1(x) &\neq 0,\  B_2(x) \neq 0,\  x \in M \backslash\{0\},
\end{align}
then $V_2(x)$ is a stochastic control Lyapunov function of the system \eqref{eq:brockett} with \eqref{eq:brockett2}--\eqref{eq:brockett8}. \eot
\end{theorem}

\begin{proof}
Randomizing the Brockett integrator \eqref{eq:brockett} by using \eqref{eq:brockett2}--\eqref{eq:brockett8}, the stochastic system
\begin{subequations}\label{eq:bro-ito}
\begin{align}
\label{eq:bro-ito1}&dx = \{ f(x) + g(x) u \}dt + g(x) B(x) dw
\end{align}
with
\begin{align}
\label{eq:bro-ito2}&f(x) = \left( 0,0,-\frac12 (b_1 b_4 - b_2 b_3)B_1(x)B_2(x) \right)^T,
\end{align}
\end{subequations}
is derived. Therefore,
\begin{align}
\label{eq:bro-diff}&(LV_2)(x) = \lie{f}{V_2}(x) + \frac12 B(x)^T H(x) B(x)
\end{align}
and
\begin{align}
&H(x) := g^T(x) \left[\pfrac{}{x}\left[\pfrac{V_2(x)}{x}\right]^T\right] g(x)
\end{align}
become
\begin{align}
(LV_2)(x) = &-2(b_1b_4-b_2b_3)B_1(x)B_2(x)x_3(x) + 2B(x)^T H(x) B(x) 
\end{align}
and
\begin{align}
H(x)=&\begin{bmatrix}
-b_1 (1+x_3^2) & 0 \\ 0 & -b_2(1+x_3^2)
\end{bmatrix}
\end{align}
in $x \in M$. If $B(x)$ is so designed such that \eqref{eq:brockett5}--\eqref{eq:brockett8} hold, then
\begin{align}
(LV_2)(x) < 0,\ x \in M
\end{align}
is obtained. Therefore, $V_2(x)$ is a stochastic control Lyapunov function of the stochastic system \eqref{eq:bro-ito}.
\end{proof}

In \tref{the:brockett}, the pre-feedback law $v(x)$ is so constructed as to make the Wong-Zakai correction terms for $dx_1$ and $dx_2$ vanish; i.e., $v(x)$ is made to simplify the design problem of the diffusion coefficient $B(x)$. \tref{the:brockett} converts the stochastic stabilization problem of the Brockett integrator \eqref{eq:brockett} into the construction problem of a stochastic stabilizing feedback law for the system \eqref{eq:bro-ito} via the stochastic control Lyapunov function $V_2(x)$. The following theorem is immediately obtained.

\begin{theorem}\label{the:sontag}
A Sontag-type controller \cite{sontag}
\begin{align}
\label{eq:sontag} u = u_s(x) :=\left\{ \begin{matrix} -\frac{F(x)+\sqrt{F(x)^2 + G(x) ^2}}{G(x)}\lie{g}{V_2}^T(x), &  G(x) \neq 0 \\
0, & G(x)=0
\end{matrix}\right.
\end{align}
with
\begin{align}
&G(x) := \lie{g}{V_2}(x) \lie{g}{V_2}(x)^T, \\
&F(x) := (LV_2)(x)
\end{align}
makes the origin of the system \eqref{eq:bro-ito} locally asymptotically stable in probability. \eot
\end{theorem}
\begin{proof}
\tref{the:brockett} yields that $(\mathcal{L}V_2)(x)<0$ while $x \in M \backslash \{0\}$. If $x \notin M \backslash \{0\}$,
\begin{align}
(\mathcal{L}V_2)(x) &= -\left\{ F(x)+\sqrt{F(x)^2 + G(x)^2}) \right\} < 0
\end{align}
is obtained by substituting $u=u_s(x)$ with $(\mathcal{L}V_2)(x)$. Therefore, the Sontag-type controller $u=u_s(x)$ locally asymptotically stabilizes the origin of the system \eqref{eq:bro-ito}.
\end{proof}

Then, the following theorem is obtained by using the Sontag-type controller \eqref{eq:sontag}.

\begin{theorem}\label{the:sscp}
Considering \tref{the:brockett}, if the diffusion coefficient $B(x)$ satisfies
\begin{align}
\label{eq:continuous1}&\lim_{x_3 \rightarrow 0} B_1(x) B_2(x) x_3 = 0, \\
\label{eq:continuous2}&\lim_{x \rightarrow 0} \frac{b_1^2 B_1(x)^2 + b_2^2 B_2(x)^2}{b_1^2 x_1^2 + b_2^2 x_2^2}=0,
\end{align}
then the stochastic control Lyapunov function $V_2(x)$ satisfies the small control property. \eot
\end{theorem}

\begin{proof}
Because $B_1(x)$ and $B_2(x)$ are once differentiable, $F(x)$ is continuous for all $x \in \R^3$. Therefore, the Sontag-type controller \eqref{eq:sontag} is continuous except the origin \cite{sontag}. The controller \eqref{eq:sontag} is $0$ in $x \in M$. For $x \in M':=\R^3 \backslash M$, 
\begin{align}
\lie{f}{V_2}(x) = &-2^{-1-x_3^2/2} B_1(x) B_2(x) (b_2 b_3-b_1 b_4) \nonumber \\
	&\times \left\{ 2^{x_3^2/2}(X(x)-4)+X(x)^{1+x_3^2/2}\log(2/X(x)) \right\} x_3 
\end{align}
is obtained. Therefore, \eqref{eq:continuous1} yields that
\begin{align}
\lim_{x_3 \rightarrow 0} \lie{f}{V_2}&(x) = 0.
\end{align}
Further, for $x \in \{x \in \R^3 | x \in M', x_3 = 0\}$,
\begin{align}
G(x) = &b_1^2 x_1^2 + b_2^2 x_2^2 \\
B(x)^T H(x) B(x) = &b_1^2 B_1(x)^2 + b_2^2 B_2(x)^2 \nonumber \\
	&- X(x) (b_4 B_2 x_1 - b_3 B_1(x) x_2)^2 \log(2/X(x)) \nonumber \\
	&- (X(x)-4)(b_4 B_2(x) x_1 - b_3 B_1(x) x_2)^2
\end{align}
is obtained. Therefore, \eqref{eq:continuous2} yields that
\begin{align}
\lim_{X \rightarrow 0} u_s(x) = 0,\ x \in M'.
\end{align}
Then, the conditions of \tref{the:florchinger} are satisfied if $k(x)=u_s(x)$. Therefore, $V_2(x)$ satisfies the small control property. 
\end{proof}

Moreover, the following corollary is derived.

\begin{corollary}\label{cor:global}
The origin of the closed-loop system \eqref{eq:bro-ito} with \eqref{eq:sontag} is globally asymptotically stable in probability. \eot
\end{corollary}
\begin{proof}
The function $V_2(x)$ is positive definite and proper in $x \in \R^3$. Further, $(\mathcal{L}V_2)(x)$ is negative definite in $x \in \R^3$ for the system \eqref{eq:bro-ito} with \eqref{eq:sontag}. Therefore, the corollary is proved.
\end{proof}


\begin{remark}
\rref{rem:kushner} and \corref{cor:global} yield that any path of the closed-loop system \eqref{eq:bro-ito} with \eqref{eq:sontag} converges to the origin with probability one. \eot
\end{remark}

\begin{remark}
The origin of the closed-loop system \eqref{eq:bro-ito} with \eqref{eq:sontag} is asymptotically stable in probability; however, it is not exponentially stable in probability because $V(x)$ is concave down in $M$. To improve the convergence rate, for example, one can apply the sliding mode controls \cite{bloch}.
\end{remark}

\section{Numerical Simulation}\label{sec:simulation}

In this section, the randomized Brockett integrator \eqref{eq:bro-ito} with $b_1=b_2=1$ and $b_3=b_4=4$ is considered. Because the randomization is operated for stochastic stabilization, the diffusion coefficient $B(x)$ should vanish while the eigenvalues of $H(x)$ are positive. Therefore, $B(x)$ is designed by
\begin{align}
&B_1 (x) = k_1 \lambda_1^2 (x_1^2+x_2^2+x_3^2), \\
&B_2 (x) = k_2 \lambda_2^2 (x_1^2+x_2^2+x_3^2) x_3,
\end{align}
where $\lambda_1,\lambda_2: \R^3 \rightarrow \R$ are the eigenvalues of $H(x)$, and $k_1>0$ and $k_2>0$ the design parameters. Fig.~\ref{fig:bro-lv} confirms that $(\mathcal{L}V_2)(x)$ is negative definite; moreover, Figs.~\ref{fig:ex1-state} and \ref{fig:ex1-input} show that the sample paths of the state $x(t)$ and the input $u(t)$ converge to $0$, respectively. The numerical simulation is calculated with the initial state $x_1(0)=x_2(0)=0$ and $x_3(0)=1$ with the design parameters $k_1=k_2=1.0 \times 10^{-4}$ using Euler-Maruyama scheme \cite{hanson}. 


\begin{figure}[t]
\begin{center}
\includegraphics[width=0.4\textwidth,keepaspectratio=true]{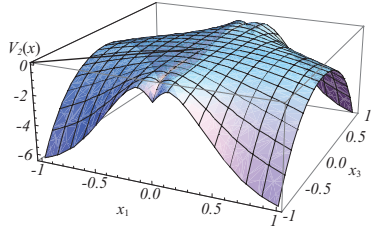}
\end{center}
\caption{$(\mathcal{L}V_2)(x)$ in $x_2=0$ with $k_1=k_2=0.2$. }
\label{fig:bro-lv}
\end{figure}
\begin{figure}[t]
\begin{center}
\includegraphics[width=0.4\textwidth,keepaspectratio=true]{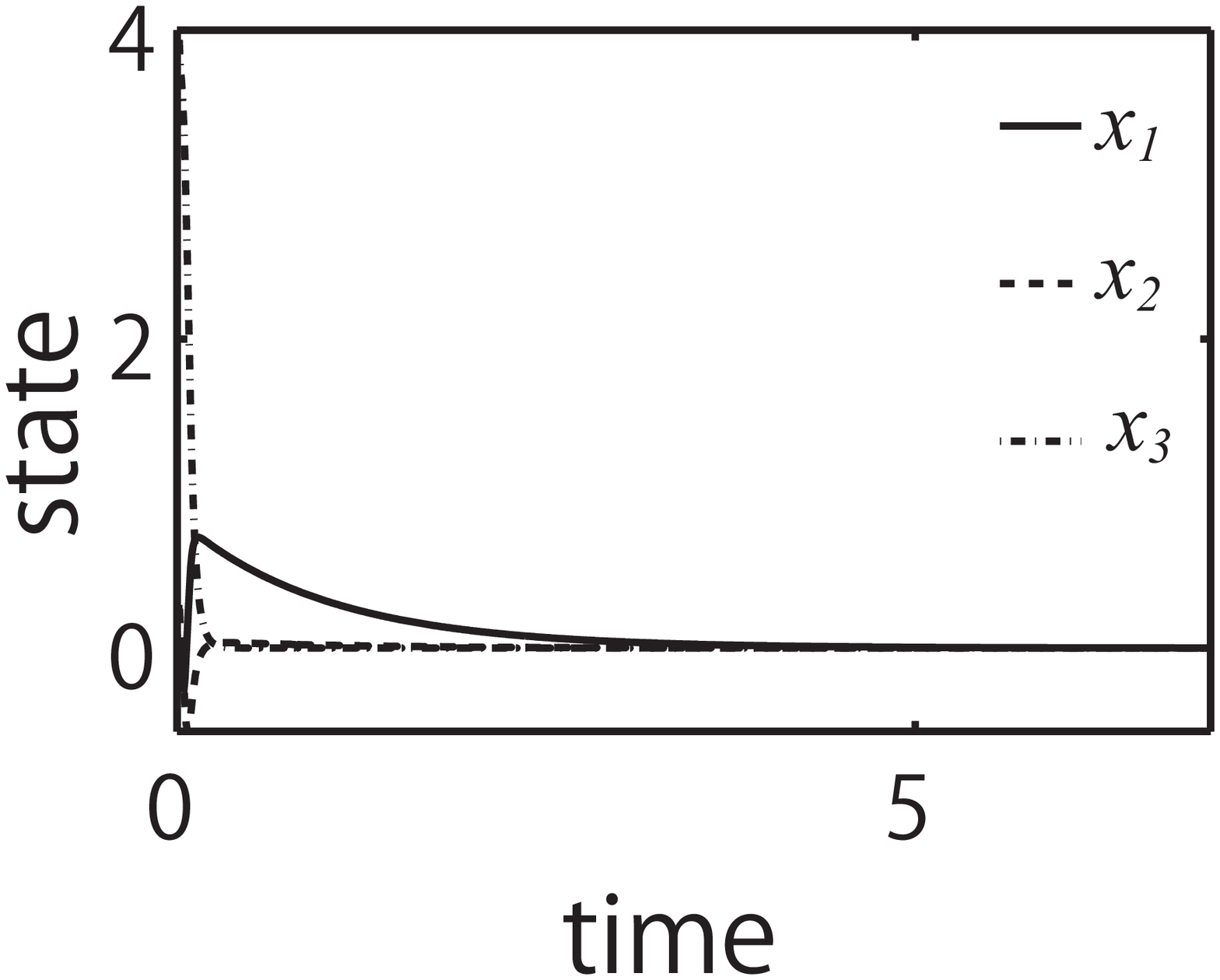}
\end{center}
\caption{Simulation result of state ($b_1=b_2=1$ and $b_3=b_4=4$).}
\label{fig:ex1-state}
\begin{center}
\includegraphics[width=0.4\textwidth,keepaspectratio=true]{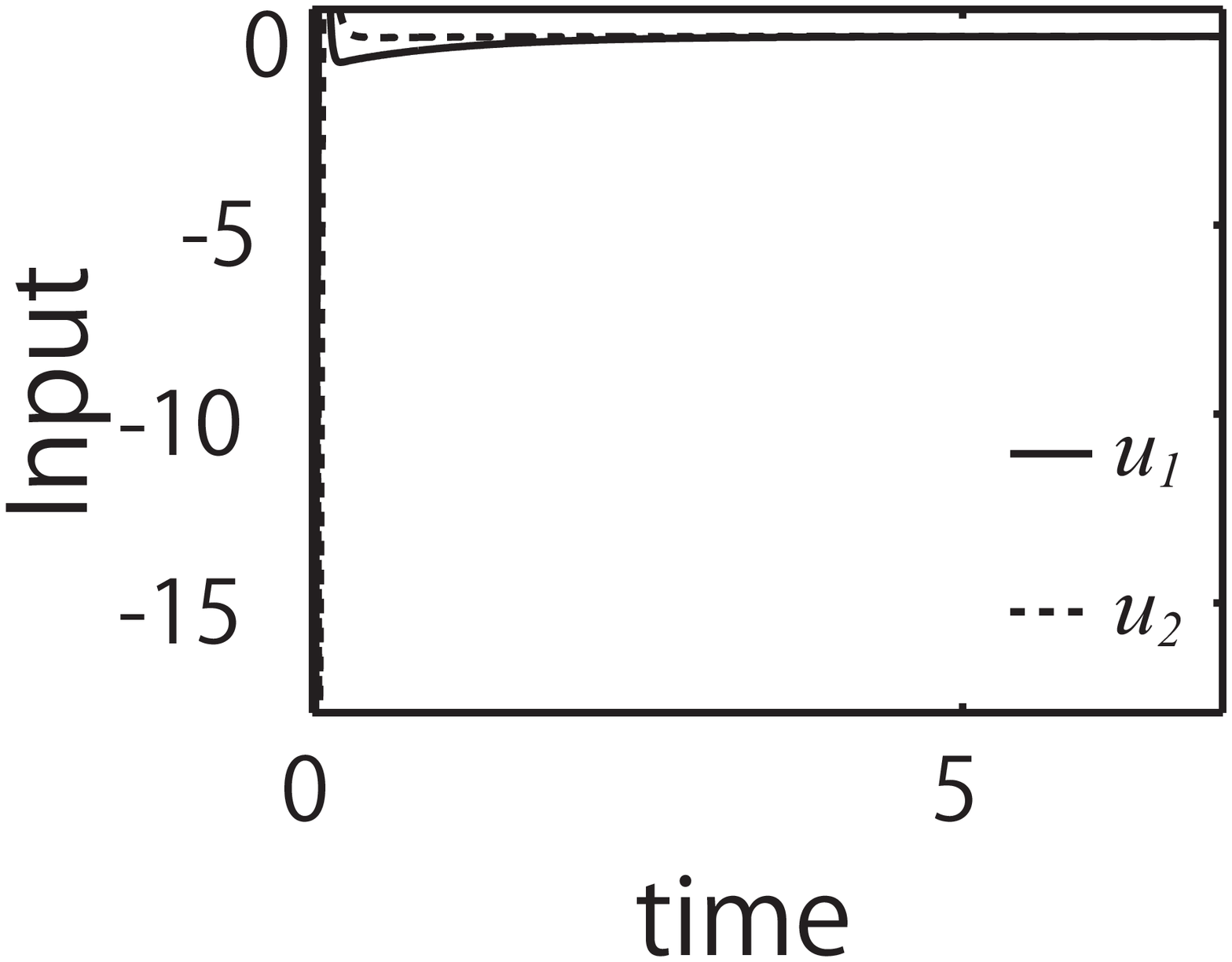}
\end{center}
\caption{Simulation result of input ($b_1=b_2=1$ and $b_3=b_4=4$).}
\label{fig:ex1-input}
\end{figure}


\section{Conclusion}\label{sec:conclusion}

In this paper, sufficient condition is proposed for the diffusion coefficients such that the origin of the input-affine systems becomes locally asymptotically stable in probability. Moreover, the stochastic continuous feedback law, which makes the origin of the Brockett integrator be globally asymptotically stable in probability, is derived.

\section*{Acknowledgments}

The authors thank Professor Pavel Pakshin for his valuable comments.


\end{document}